\documentclass[12pt]{amsart}
\usepackage{enumerate, graphicx, amssymb, amsmath, amsthm}
\usepackage{mathrsfs}
\newtheorem{theorem}{Theorem}[section]
\newtheorem{lemma}[theorem]{Lemma}

\newtheorem{corollary}[theorem]{Corollary}
\newtheorem{remark}{Remark}[section]

\newcommand{\Mat}[2][cccccccccccccccccccccccccccccccccccccc]{\left[ \begin{array}{#1}#2 \\ \end{array} \right]}
\newcommand{\R}{\mathbb{R}}
\newcommand{\C}{\mathbb{C}}

\def\ve{\varepsilon}

\begin{document}
\title[Loss of derivatives for hyperbolic boundary problems]{Loss of derivatives for hyperbolic boundary problems with constant coefficients}
\author{Matthias Eller}
\address{Department of Mathematics and Statistics, Georgetown University, Washington, DC 20057, USA}
\begin{abstract}
 Symmetric hyperbolic systems and constantly hyperbolic systems with constant coefficients and a boundary condition which satisfies a weakened form of the Kreiss-Sakamoto condition are considered. A well-posedness result is established which generalizes a theorem by Chazarain and Piriou for scalar strictly hyperbolic equations and non-characteristic boundaries \cite{cp72}. The proof is based on an explicit solution of the boundary problem by means of the Fourier-Laplace transform. As long as the operator is symmetric, the boundary is allowed to be characteristic.
\end{abstract}
\maketitle
\section{Introduction}
 We consider a hyperbolic system of first order with constant coefficients for a vector-valued function $u=(u_1,u_2,...,u_N)^T$ depending on time $t$ and space $x$. This is a linear system of differential equations of the form
 \begin{equation}\label{1}
  Pu:= A^0\frac{\partial u}{\partial t}(t,x) + \sum_{j=1}^d A^j\frac{\partial u}{\partial x_j}(t,x)  = f(t,x), \qquad (t,x) \in \R^{d+1}_+\equiv \R_t \times \R^d_+
 \end{equation}
 Here $\R^d_+=\{x=(x_1,...,x_d)\;:\; x_d>0\}$ is a half space, the coefficients $A^j$ are $N\times N$ matrices. The system is assumed to be either symmetric hyperbolic in which case the matrices $A^0,A^1,...,A^d$ are Hermitian and the matrix $A^0$ is positive definite, or constantly hyperbolic which is to say that the polynomial
 \[
  \mbox{det}\,P(\tau,\xi) = \mbox{det}\left[A^0\tau + \sum_{j=1}^d A^j \xi_j  \right]
 \]
 has only geometrically regular real roots in $\tau$ of constant multiplicity for all $\xi\in S^{d-1}$. The notion of constant multiplicity is a slight generalization of the concept of strict hyperbolicity where the zeros of the same polynomial are real and simple. In the non-symmetric case the boundary $\partial \R^{d+1}_+=\{(t,x)\in \R^{d+1}\;:\; x_d=0\}$ is assumed to be non-characteristic which is to say that the matrix $A^d$ is non-singular. If the matrix $A^d$ is singular, the boundary $\partial \R^{d+1}_+$ is characteristic.

 The function $f(t,x)$ is a vector-valued function in $\R_+^{d+1}$ with $N$ components. The differential equation (\ref{1}) will be complemented by a boundary condition of the form $Bu(t,y,0) = g(t,y)$ on where $y = (x_1,...,x_{d-1})$, $B$ is a $p\times N$ matrix, and $g$ is vector-valued function in $\R^d$ with $p$ components. In order to obtain the correct number of boundary conditions, that is, to determine $p$, we consider the system of ordinary differential equations
 \begin{equation}\label{2}
  A^d\frac{\partial v}{\partial x_d} = -i\left[A^0(\tau-i\gamma) + \sum_{j=1}^{d-1} A^j\eta_j \right]v
 \end{equation}
 for $x_d>0$ where $(\tau,\gamma,\eta_1,...,\eta_{d-1})$ are real parameters and $i=\sqrt{-1}$. Only exponentially decreasing solution in $x_d$ are of interest since our analysis will take place in the context of square-integrable functions. For fixed $(\tau,\eta,\gamma)\in \R^d \times (0,\infty)$ the solution of this system gives rise to an exponentially growing solution to the original system with $f\equiv 0$ of the form
 \[
  u(t,x) = e^{i(\tau-i\gamma)t + i\eta\cdot y}v(x_d)\;.
 \]
 Here and henceforth $\eta = (\eta_1,\eta_2,...,\eta_{d-1})\in \R^{d-1}$. The boundary condition $B$ will serve to exclude solutions which are both exponentially increasing in time and exponentially decreasing in $x_d$. In the following we will discuss this condition as has been done before in the non-characteristic case, see for example \cite{he63,gav02}.

 Due to homogeneity reasons it will suffice to restrict ourselves to the case $|\tau-i\gamma|^2+|\eta|^2=\tau^2+\gamma^2+|\eta|^2=1$, where $|\cdot|$ denotes the Euclidean norm in the vector space $\C^m$. For brevity we write
 \[
  S^d_+ = \{ (\tau,\eta,\gamma) \in \R^{d+1}\;:\; \tau^2+|\eta|^2+\gamma^2=1,\;\gamma>0\}.
 \]
 Hyperbolicity implies that the matrix
 \begin{equation}\label{7}
  G(\tau-i\gamma,\eta) = -i\left[A^0(\tau-i\gamma) + \sum_{j=1}^{d-1} A^j\eta_j \right]
 \end{equation}
 is non-singular for all $(\tau,\eta,\gamma)\in S^d_+$ and that the matrix $A^d$ is diagonalizable with real eigenvalues. If $A^d$ is singular, then equation (\ref{2}) is a singular system of ordinary differential equations. For a comprehensive theory on singular systems we refer to the book by Campbell \cite{campbell80}. Since we are interested in well-posedness results for hyperbolic systems in the $L_2$-framework, we will be interested in square integrable solutions to equation (\ref{2}). The matrix pair $(A^d,G)$ is said to be regular, if there exists a pair $(\alpha,\beta)$ of complex numbers such that $\mbox{det}[\alpha A^d -\beta G]\neq 0$. Note that the matrix pair $(A^d,G)$ is regular since $G$ is invertible. A vector $z\in \C^N$ is an eigenvector with eigenpair $(\alpha,\beta)$ if
 \begin{equation}\label{18}
  \alpha A^d z = \beta G z\;.
 \end{equation}
 If $(\alpha,\beta)$ satisfies equation (\ref{18}), so does $(s\alpha,s\beta)$ for all $s\in \C$. Corresponding to each eigenpair $(\alpha,\beta)$ we have the eigenvalue $\lambda = \alpha/\beta$ and we write $\lambda =\infty$ for $(\alpha,\beta)=(1,0)$.

 The regular pair $(A^d,G)$ has $N$ eigenvalues counting multiplicities and there exists a basis for $\C^N$ consisting  of generalized eigenvectors \cite[p. 275]{stewartsun90}.  Let $E^s$ be the stable subspace of $\C^N$ corresponding to all eigenvalues $\lambda$ with $\Re \lambda <0$, let $E^u$ be the unstable subspace corresponding to all eigenvalues of the form $\lambda$ with $\Re \lambda>0$, and let $E^c$ be the central subspace corresponding to all infinite eigenvalues. Note that $E^c = N(A^d)$ is the null space of $A^d$ since $A^d$ is diagonalizable. These three subspaces are invariant subspaces of the matrix $G^{-1}A^d$.

 Since $P$ is hyperbolic, there are no purely imaginary eigenvalues. Hence, the dimensions of the stable subspace and the unstable subspace are independent of $(\tau,\eta,\gamma)\in S^d_+$ and we define $\mu = \mbox{dim}\, E^s$. Connecting this analysis with the argument above we realize that we need $p=\mu$ and that $Bv=0$ for $v\in E^s$ implies $v=0$. 

 A uniform version of this condition in the non-characteristic case is the Kreiss-Sakamoto (uniform Lopatinskii) condition. A boundary operator $B$ is a $\mu\times N$ matrix such that
 \begin{equation}\label{6}
  |Bv| \gtrsim |v|
 \end{equation}
 for all $v\in E^s(\tau-i\gamma,\eta)$ where $(\tau,\eta,\gamma)\in S^d_+$. Here and henceforth $a\lesssim b$ means $a\le Cb$ for some constant $C$ independent of $(\tau,\gamma,\eta)$. Note that this condition implies that $B$ has rank $\mu$. Already in 1970, Kreiss and Ralston showed that the hyperbolic boundary problem
 \begin{equation}\label{3}
  Pu = f \mbox{ in } \R^{d+1}_+,\qquad Bu = g \mbox{ in } \{x_d=0\}\;,
 \end{equation}
 is strongly well-posed in $L_2$ for strictly hyperbolic $P$ in the non-characteristic case if and only if $B$ satisfies the Kreiss-Sakamoto condition \cite{kr70,ra71}. Strong well-posedness in $L_2$ means that for functions $f\in L_2(\R^{d+1}_+)$,  $g\in L_2(\R^d)$ which vanish for $t<0$, there exists a unique weak solution $u\in L_{2,loc}(\R^{d+1}_+)$ vanishing for $t<0$ and the estimate
 \begin{multline}\label{4}
  \gamma \int_{\R^{d+1}_+}e^{-2\gamma t} |u(t,x)|^2 dtdx + \int_{R^d}e^{-2\gamma t}|A^d u(t,y,0)|^2dtdy \\ \lesssim \frac{1}{\gamma}\int_{R^{d+1}_+}e^{-2\gamma t}|f(t,x)|^2\,dtdx + \int_{\R^d}e^{-2\gamma t}|g(t,y)|^2\,dtdy
 \end{multline}
 holds for $\gamma$ sufficiently large. This estimate includes the additional regularity result $e^{-\gamma t} A^d u(t,y,0) \in L_2(\R^d)$. Of course, in the non-characteristic case the matrix $A^d$ in the second integral on the left-hand side may be omitted.

 Sakamoto obtained a similar result for scalar strictly hyperbolic equations of arbitrary order \cite{sakamoto82}. Many interesting boundary conditions, for example the Neumann boundary condition or the oblique derivative condition for the scalar wave equation, do not satisfy the Kreiss-Sakamoto condition. These boundary conditions satisfy the Lopatinskii condition: $Bv=0$ implies $v=0$ for all $v\in E^s(\tau-i\gamma,\eta)$ where $(\tau,\eta,\gamma)\in S^d_+$ and $N(A^d)\subset N(B)$. If the Lopatinskii condition holds, then the hyperbolic boundary problem has a unique solution, at least in the non-characteristic case \cite{he63}. However, this does not imply an estimate of the form (\ref{4}).

In this note we will investigate a weakened form of the Kreiss-Sakamoto condition. We say that the $\mu\times N$ matrix $B$ satisfies the Kreiss-Sakamoto condition with power $s$ if $N(A^d)\subset N(B)$ and
 \begin{equation}\label{5}
  |Bv| \gtrsim \gamma^s|A^d v|
 \end{equation}
 for all $v\in E^s(\tau-i\gamma,\eta)$ where $(\tau,\eta,\gamma)\in S^d_+$ and some $s \in [0,1]$. In the case $s =0$ this condition is equivalent to Assumption 1.3. in the classical paper by Majda and Osher \cite{mo75} on the characteristic problem. We will show that our condition results in a well-posedness result, albeit with a loss of regularity. Our main result is the following theorem.
 \begin{theorem}\label{A}
  Consider the hyperbolic boundary problem (\ref{3}) where $P$ is either symmetric hyperbolic or constantly hyperbolic with non-singular $A^d$. Then, the following two statements are equivalent. \\
  (i) The boundary operator $B$ satisfies the Kreiss-Sakamoto condition with power $s$. \\
  (ii) For $f\in L_2(\R_+,H^s(\R^d))$ and $g\in H^s(\R^d)$ supported in $t\ge 0$, there exists a unique weak solution $u\in L_{2,loc}(\R^{d+1}_+)$ to the boundary problem (\ref{3}), satisfying $A^d u(0) \in L_{2,loc}(\R^d)$ and vanishing for $t<0$, and the estimate
  \begin{equation}\label{10}
   \gamma \|e^{-\gamma t}u\|^2_0 + |e^{-\gamma t}A^d u(0)|^2_0 \lesssim \frac{1}{\gamma^{1+2s}}\|e^{-\gamma t}f\|^2_{s,\gamma} + \frac{1}{\gamma^{2s}}|e^{-\gamma t}g|^2_{s,\gamma}
  \end{equation}
  holds for $\gamma$ sufficiently large.
 \end{theorem}
 Here we use the Sobolev norms
 \[
  \begin{split}
   |g|^2_{s,\gamma} &= \int_{\R^d}(\gamma^2+|\eta|^2+\tau^2)^s |\hat{g}(\tau,\eta)|^2d\tau d\eta \\
   \|f\|^2_{s,\gamma}&=\int_0^\infty \int_{\R^d}(\gamma^2+|\eta|^2+\tau^2)^s |\hat{f}(\tau,\eta,x_d)|^2d\tau d\eta dx_d
  \end{split}
 \]
 for $\gamma>0$, where $\hat{g}(\tau,\eta)$ is the Fourier transform of $g(t,y)$ and $\hat{f}(\tau,\eta,x_d)$ is the tangential Fourier transform of $f(t,y,x_d)$, that is the Fourier transform with respect to $(\tau,y)$. The norms $\|\cdot\|_0$ and $|\cdot|_0$ are the $L_2$-norms in the half space $\R^{d+1}_+$ and its boundary $\{x_d=0\}$, respectively. Comparing the estimate (\ref{10}) with (\ref{4}) one notices the higher norms on the right-hand side in (\ref{10}). The solution of the boundary problem is only of regularity $L_2$ in the interior and on the boundary where the data $f$ and $g$ have $s$ (tangential) derivatives. Hence, the estimate (\ref{10}) is characterized by the loss of $s$ derivatives in the interior and along the boundary.

 For the wave equation with Neumann boundary condition we will show that $s=1/2$ whereas for the oblique derivative problem one has $s=1$, see also \cite[Section 11]{cp72}. In this connection we like to mention that in the case of this Neumann problem (and for conservative boundary conditions in general) the loss of derivatives in the interior can be avoided. There are estimates similar to (\ref{10}) where $f$ and $u$ have the same regularity \cite[Theorems 1.2,1.3]{el12}. In contrast to our work, these results hold even in the case of variable coefficients but a loss of one derivative occurs on the boundary .

 Theorem \ref{A} complements an earlier result given by Chazarain and Piriou for scalar strictly hyperbolic equations in the non-characteristic case \cite{cp72}. To our best knowledge, the only result in this direction for systems has been proved by Coulombel in the case $s=1$ under rather restrictive additional assumptions \cite[Theorem 2.1]{coulombel04}. This is due to the fact that Coulombel bases his analysis on the Kreiss symmetrizer. We obtain a less restrictive result here by solving the constant coefficient problem by means of the Fourier-Laplace transform.

 However, in contrast to Coulombel, our method does not carry over to variable coefficients, see \cite[Theorem 3.2]{coulombel04}. Furthermore, our method does not allow terms of zero order be incorporated into the operator $P$.  Nevertheless we expect that Theorem \ref{A} remains true in the special case of constantly hyperbolic operators with non-characteristic boundary, where the Kreiss-Sakamoto condition (\ref{6}) is violated only in the hyperbolic region. This is the set of frequencies $(\tau,\eta)\in S^{d-1}$ where the matrix $[A^d]^{-1}G(\tau,\eta)$ is diagonalizable with purely imaginary eigenvalues. For a thorough discussion of the hyperbolic region we refer to a recent work by Coulombel \cite{co11}. These boundary problems have been characterized as {\it weakly regular of real type} by Benzoni-Gavage, Rousset, Serre, and Zumbrun \cite[Section 2.4]{gav02}.

 Even in the case of the Kreiss-Sakamoto condition ($s=0$), Theorem \ref{A} produces $L_2$ well-posedness for symmetric hyperbolic problems.
 \begin{corollary}\label{B}
  Suppose that $P$ is symmetric hyperbolic and that the boundary operator $B$ satisfies the Kreiss-Sakamoto condition (\ref{6}). Then the boundary problem is well-posed in $L_2$, i.e.  for $f\in L_2(\R^{d+1}_+)$ and $g\in L^2(\R^d)$ vanishing for $t<0$, there exists a unique solutions $u\in L_{2,loc}(\R^{d+1}_+)$, vanishing for $t<0$ and the estimate (\ref{4}) holds.
 \end{corollary}
 This result has been established in the book by Serre \cite{serre00} in the case $f\equiv 0$ and only rather recently in the non-characteristic case  by Gu\`es, M\'etivier, Williams, and Zumbrun \cite{gmwz07} where this result is discussed in detail. This corollary extends their result to the characteristic case. We like to point out that the classical work on the characteristic boundary value problem by Majda and Osher requires some additional assumptions \cite[Section 1]{mo75}. Our result shows that these assumptions are not needed, at least in the constant coefficient case. On the other hand, the work by Majda and Osher shows that in the characteristic case the operator needs to be symmetric \cite[Exampe B.1]{mo75}.

 For characteristic problem the condition $N(A^d)\subset N(B)$ has been used since the work by Majda and Osher. Ohkubo refers to this condition as reflexive \cite{oh81}. Example B.2 in the work of Majda and Osher \cite[Section 2]{mo75} shows that the estimate (\ref{4}) may not hold if this condition is violated.
 \begin{remark}
  Both Theorem \ref{A} and Corollary \ref{B} remain valid in the case that the matrix $G$ is a homogeneous function in $(\tau,\eta,\gamma)$ of degree one. This is of some relevance when considering higher-order systems or equations since in this case the reduction to a system of first order is performed using Fourier multipliers and not only differential operators. For an example we refer to section \ref{iv}.
 \end{remark}
  Theorem \ref{A} will be proved in Section \ref{iii}. Strictly speaking we prove only that (i) implies (ii). The opposite direction is rather straightforward, see for example \cite{kr70,cp72,gmwz07}. Our proof follows the approach by Chazarain and Piriou given in 1972 up to a certain point. The boundary value problem is solved explicitly  by means of the tangential Fourier transform.  This solution is estimated relying on ideas given in the paper \cite{gmwz07}. It is an interesting fact that Kreiss avoided precisely this approach since it does not generalize to variable coefficients \cite[p.281]{kr70}.

  We feel that the analysis of the constant coefficient problem for hyperbolic systems has its merits. Precise estimates for boundary value problems with boundary conditions where the Kreiss-Sakamoto condition does not hold are obtained. In particular, the example of Maxwell's equations given in section \ref{iv} shows that the regularity statements given by the author in the context of conservative boundary conditions \cite{el12} is sharp. More importantly, a recent preprint by G. M\'etivier contains an example of a symmetric hyperbolic system with variable coefficients and a Kreiss-Sakamoto boundary condition which fails to be well-posed in $L_2$ \cite{me14}.

  The author wishes to thank Professor Kevin Zumbrun (University of Indiana) for pointing out the reference \cite{gmwz07}. Moreover, the author is very grateful and very appreciative to the two anonymous referees who gave a number of helpful comments and pointed out a mistake in the original proof.

\setcounter{equation}{0}
\section{Proof of Theorem \ref{A}}\label{iii}
 Using a density argument it will suffice to work with smooth functions $f$ and $g$ with compact support in $\overline{\R^{d+1}_+}$ and $\R^{d-1}$, respectively which vanish for $t<0$. For the reader's convenience we will structure the proof into three steps. In the first step we perform the Fourier transform in the tangential variables and solve the resulting boundary value problem for an ordinary differential equation (or in the characteristic case a differential-algebraic equation) with parameters. In the following subsection we obtain $L_2$-estimates for these explicit solutions with boundary traces which do not involve the boundary condition $B$ yet. However, these estimates involve a boundary condition which satisfies the Kreiss-Sakamoto condition with power zero and they are uniform with respect to the parameters. At this point we are forced to consider the symmetric case and the constantly hyperbolic case separately. In the final subsection we introduce the boundary condition $B$ and obtain the estimates corresponding to Theorem \ref{A}. In this part we follow largely the approach given in \cite[Prop. 2.6]{gmwz07}
\subsection{An explicit solution to the boundary value problem.}
 For $\gamma>0$ we multiply the differential equation by $e^{-\gamma t}$ and obtain
 \[
  P(e^{-\gamma t} u) + \gamma A^0 e^{-\gamma t}u = e^{-\gamma t} f ,\qquad \mbox{ for } x_d>0\;,
 \]
 and the boundary condition is $B e^{-\gamma t}u = e^{-\gamma t}g$ at $x_d=0$. After performing a Fourier transform in the tangential variables $(t,y)$ one obtains the ordinary differential equation
 \begin{equation}\label{9}
  A^d\frac{\partial \hat{u}}{\partial x_d}(\tau-i\gamma,\eta,x_d) = G(\tau-i\gamma,\eta)\hat{u}(\tau-i\gamma,\eta,x_d) + \hat{f}(\tau-i\gamma,\eta,x_d)
 \end{equation}
 with the initial condition
 \begin{equation}\label{8}
  B\hat{u}(\tau-i\gamma,\eta,0) = \hat{g}(\tau-i\gamma,\eta)\;.
 \end{equation}
 The matrix $G$ has been introduced in formula (\ref{7}). We start by proving the following resolvent type estimate, see \cite[Lemma 2.2]{kr70} For $N\times N$ matrices we will use the spectral norm $|\cdot|_S$.
\begin{lemma}
 Under the assumptions on the hyperbolic operator $P$ stated in the introduction we have
 \[
  |(P -i\gamma  A^0)^{-1}|_S= |(i\xi_d A^d - G(\tau-i\gamma ,\xi_1,...,\xi_{d-1}))^{-1}|_S \lesssim \frac{1}{\gamma}
 \]
 for all $(\tau,\xi) \in \R^{d+1}$ and $\gamma>0$.
\end{lemma}
\begin{proof}
 In the symmetric hyperbolic case we may assume without loss of generality that $A^0=I_N$, the $N\times N$ identity matrix. Otherwise we premultiply and postmultiply the operator by the inverse of the Hermitian square root of $A^0$. There exists a unitary matrix $Q=Q(\xi)$ which diagonalizes $P$,
\[
  Q^H [ P -i\gamma I_N ] Q = \Mat{ \tau - i\gamma +\tau_1(\xi) & 0 & \cdots & 0 \\ 0 & \tau-i\gamma+\tau_2(\xi) & \cdots & 0 \\  \vdots & \vdots & \ddots & \vdots \\ 0 & 0& \cdots & \tau-i\gamma + \tau_N(\xi) }\;,
\]
where the $\tau_j(\xi)$ are real-valued for $j=1,2,..,N$. Of course, the matrix $Q$ may not be smooth. Nevertheless, all eigenvalues of $(P -i\gamma I_N) (P -i\gamma I_N)^H$ are greater or equal to $\gamma^2$. Hence all eigenvalues of the $(P -i\gamma I_N)^{-H} (P -i\gamma I_N)^{-1}$ are less or equal to $\gamma^{-2}$. Here $A^{-H}$ denotes the inverse of the Hermitian transpose.

In the constantly hyperbolic case, the matrix $P$ premultiplied by $A^0$ can be diagonalized by means of a smooth matrix $Q(\xi)$ which is homogeneous of degree zero in $\xi$. This implies that the matrix $Q$ and its inverse have a spectral norm which is bounded below and above, independent of $\xi$. Hence, the proof in this case is very similar to the symmetric case and details will be omitted.
\end{proof}

For brevity we will drop the parameter $\tau-i\gamma$ and $\eta$, even though we keep in mind that the solution we construct will depend smoothly on $(\tau,\eta,\gamma)$.  A square integrable solution $U_1$ to the differential equation (\ref{9}) is given by
 \begin{equation}\label{11}
  U_1(x_d) = \int_{\R}e^{ix_d\xi_d}(i\xi_d A^d - G)^{-1} \tilde{f}(\xi_d) d\xi_d\;,
 \end{equation}
 where $\tilde{f}$ denotes the Fourier transform of the function $\hat{f}$ with respect to $x_d$. Note that the Lemma above shows in connection with Parseval's identity the estimate
 \[
  \|U_1\|_{L_2(\R_+)} \lesssim \frac{1}{\gamma} \|\hat{f}\|_{L_2(\R_+)}
 \]
 uniformly in $(\tau,\eta,\gamma)\in \R^d_{\tau,\eta}\times \R_+$.

 In order to obtain a function $U$ which satisfies the differential equation (\ref{9}) and the initial condition (\ref{8}), we define the matrices
 \[
   \Pi_{E^s}=\frac{1}{2\pi i} \int_{\Gamma_-}(\zeta A^d - G)^{-1}A^d d\zeta\quad  \mbox{ and } \quad J = B\Pi_{E^s}\;.
 \]
 Here $\Gamma_-$ is a smooth, simple, closed, and positively oriented curve in the complex plane which encloses all eigenvalues of the matrix pair $(A^d,G)$ with a negative real part and excludes all eigenvalues with positive real part.
 The first matrix is the spectral projection onto the stable subspace $E^s$ of the matrix pair $(A^d,G)$ and the latter one is of rank $\mu$ since the matrix $B$ satisfies the Lopatinskii condition. Set
 \begin{equation}\label{13}
  U(x_d) = U_1(x_d) + U_2(x_d) \;,
 \end{equation}
 where
 \begin{equation}\label{31}
  U_2(x_d) =  \frac{1}{2\pi i} \int_{\Gamma_-}e^{x_d\zeta}(\zeta A^d - G)^{-1}A^d d\zeta J^H(JJ^H)^{-1}[\hat{g}-BU_1(0)]
 \end{equation}
 and $J^H$ is the Hermitian transpose of $J$. One checks that $U_2$ is a solution to the equation (\ref{9}) with $\hat{f}\equiv 0$ and that
 \begin{multline*}
  BU(0) = BU_1(0) + BU_2(0) = BVU_1(0) + J[J^H(JJ^H)^{-1}(\hat{g}-BU_1(0))]\\ = BU_1(0)+\hat{g}-  BU_1(0)=\hat{g} \;.
 \end{multline*}
 Observe that $U \in L_2(\R_+)$ since $U_1\in L_2(\R_+)$ and $U_2$ is exponentially decaying as $x_d \to \infty$. Furthermore, since $G$, $\hat{f}$, and $\hat{g}$ are analytic functions in the complex variable $\tau-i\gamma$ on the lower half plane, the same is true for our solution $U$.

 This function $U$ is the only square integrable weak solution to the boundary value problem (\ref{9}), (\ref{8}).
 Suppose that $Y \in L_2(\R_+)$ is another solution. Then, the difference $Z=U-Y$ is square integrable and solves the initial value problem (\ref{9}), (\ref{8}) with $\hat{f}\equiv 0$ and $\hat{g}=0$. Hence, $Z(x_d) = (2\pi i)^{-1}\int_{\Gamma_-}e^{\zeta x_d}(\xi A^d-G)^{-1}A^d d\zeta Z(0)$ and $Z(0) \in E^s$. Since $BZ(0) = 0$, the Lopatinskii condition gives $Z(0)=0$, and hence $Z\equiv 0$.

\subsection{$L_2(\R_+^d)$ estimates with boundary traces.} The next lemma gives  $L_2$ estimates for the solutions constructed above. We have to distinguish two cases: One is for symmetric hyperbolic systems and the other for constantly hyperbolic systems. In the symmetric case we may assume that $A^d$ has the format
\[
 A^d = \left[ \begin{matrix} 0  & 0& 0 \\ 0 & A_+ & 0 \\ 0&0& -A_-  \end{matrix} \right]
\]
where $A_+$ and $A_-$ are Hermitian positive definite matrices, the matrix $A_+$ has rank $\mu$. Correspondingly, any vector $U\in \C^N$ can be written as $U = (U_0,U_+,U_-)^T$.

In the constantly hyperbolic case we use the fact that the stable subspace $E^s(\tau - i\gamma,\eta)$ is a smooth vector bundle for $(\tau,\eta,\gamma) \in S^d_+$ with a continuous extension to $\overline{S}^d_+$. This fact is proved in the strictly hyperbolic case in \cite[Chapter 7, Theorem 3.5]{chazarainpiriou82} and extends to the constantly hyperbolic case in view of M\'etivier's work \cite{me00}, see also \cite[Corollary 5.2]{mz05}. Hence, there exists a continuous decomposition
\[
 \C^N = E^s(\tau-i\gamma,\eta) \oplus E(\tau-i\gamma,\eta)
\]
for all $(\tau,\eta,\gamma) \in \overline{S}^d_+$ and the projector $Q=Q(\tau-i\gamma,\eta)$ onto $E^s$ parallel to $E$ is bounded. In contrast, the spectral projection $\Pi_{E^s}$ is not necessarily bounded as $\gamma\to 0$ since the stable subspace and the unstable subspace may collide as $\gamma$ approaches zero.
\begin{lemma}\label{F}
 Suppose $P$ is symmetric hyperbolic. Then, there exists a constant $\underline{\gamma}>0$ such that for $\gamma\ge \underline{\gamma}$ and $(\tau,\eta)\in \R^d$
 the inequality
 \begin{multline*}
  \gamma \int_0^\infty |U(\tau-i\gamma,\eta,x_d)|^2 dx_d+ |A^dU(\tau-i\gamma,\eta,0)|^2\\ \lesssim \frac{1}{\gamma}\int_0^\infty |\hat{f}(\tau-i\gamma,\eta,x_d)|^2 dx_d + |A_+ U_+(\tau-i\gamma,\eta,0)|^2
 \end{multline*}
 holds. 
 
 In the constantly hyperbolic case, there exists a constant $\underline{\gamma}>0$ such that for $\gamma\ge \underline{\gamma}$ and $(\tau,\eta)\in \R^d$ the inequality
 \begin{multline*}
  \gamma \int_0^\infty |U(\tau-i\gamma,\eta,x_d)|^2 dx_d+ |U(\tau-i\gamma,\eta,0)|^2\\ \lesssim \frac{1}{\gamma}\int_0^\infty |\hat{f}   (\tau-i\gamma,\eta,x_d)|^2 dx_d + |Q U(\tau-i\gamma,\eta,0)|^2
 \end{multline*}
 takes place.
\end{lemma}
\begin{proof}
 Suppose that $P$ is symmetric hyperbolic. We know that $\hat{f}(\tau-i\gamma,\eta,\cdot)$ is square integrable in $x_d$. Taking the real part of the inner product of $U$ with the equation
 \[
  A^d \frac{\partial U}{\partial x_d} = -i \left[ A^0\tau + \sum_{j=1}^{d-1} A^j\eta_j \right]U -\gamma A^0U + \hat{f}
 \]
 in $L_2(0,\infty)$ gives
 \[
  \frac{1}{2}\Re \int_0^\infty \frac{d}{dx_d}[U^HA^d U] dx_d = -\gamma \int_0^\infty U^HA^0 U dx_d + \Re\int_0^\infty U^H \hat{f} dx_d\;.
 \]
 Integrating by parts on the left-hand side, applying the positive definiteness of $A^0$ in the first term on the right-hand side results in
 \[
  \gamma \int_0^\infty |U|^2 dx_d + |A_- U_-(0)|^2 \lesssim \Re\int_0^\infty U^H \hat{f} dx_d + |A_+U_+(0)|^2 \;.
 \]
An application of the Cauchy-Schwarz inequality in the first term on the right-hand side finishes the proof in the symmetric hyperbolic case.

 In the case of constantly hyperbolic operators with $\det A^d \neq 0$, the differential equation (\ref{9}) can be symmetrized by means of a Fourier multiplier $S(\tau,\gamma,\xi)$ which is bounded and homogeneous in $(\tau,\eta,\gamma)\in \R^d \times \R_+$ of degree zero and has the properties
 \begin{itemize}
  \item[(i)] $S^H = S$,
  \item[(ii)] $\Re \left (S [A^d]^{-1} G(\tau-i\gamma,\eta)\right) \gtrsim \gamma$,
  \item[(iii)] There exists a positive constant $C>0$ such that $S+ C Q^H Q \gtrsim I_N$,
 \end{itemize}
   \cite[Corollary 5.2]{mz05}. The estimate is obtained by applying the symmetrizer $S$ to the equation
 \[
  \frac{\partial U}{\partial x_d} = [A^d]^{-1}GU  + [A^d]^{-1}\hat{f}
 \]
 and then taking the real part of the inner product of the resulting equation with $U$ in $L_2(0,\infty)$. Relying on (i) and (ii), one arrives at
 \[
  \frac{1}{2}\Re \int_0^\infty \frac{d}{dx_d}[U^HS U] dx_d - \Re\int_0^\infty U^H S[A^d]^{-1}\hat{f} dx_d \gtrsim \gamma \int_0^\infty |U|^2 dx_d \;.
 \]
 The proof is finished by applying property (iii) of the symmetrizer to the first term on the left-hand side and the Cauchy-Schwarz inequality and the boundedness of $S$ to the second term on the left-hand side.
\end{proof}
 We like to point out that the first estimate of the Lemma shows that the $\mu \times N$ matrix $\left[ \begin{matrix} 0 & A_+ & 0 \end{matrix} \right]$ satisfies the Kreiss-Sakamoto condition with power $s=0$.
\subsection{The estimate in the Fourier domain.} In this subsection we will apply Lemma \ref{F} and derive the
 estimate
 \begin{multline}\label{12}
  \gamma \int_0^\infty |U(\tau-i\gamma,\eta,x_d)|^2 dx_d +|A^d U(\tau-i\gamma,\eta,0)|^2 \\ \lesssim \frac{(\tau^2+\gamma^2+|\eta|^2)^s}{\gamma^{1+2s}}\int_0^\infty |\hat{f}(\tau-i\gamma,\eta,x_d)|^2 dx_d + \frac{(\tau^2+\gamma^2+|\eta|^2)^s}{\gamma^{2s}}|\hat{g}(\tau-i\gamma,\eta)|^2
 \end{multline}
 for $\gamma\ge \underline{\gamma}$ which implies the estimate of Theorem \ref{A} in view of Parseval's identity.

Because of Lemma \ref{F} it will suffice to establish
\begin{equation}\label{30}
 |A^d_+U_+(0)|^2 \lesssim \frac{(\tau^2+\gamma^2+|\eta|^2)^s}{\gamma^{1+2s}}\|\hat{f}\|_{L_2(\R_+)}^2 + \frac{(\tau^2+\gamma^2+|\eta|^2)^s}{\gamma^{2s}}|\hat{g}|^2
\end{equation}
in the symmetric case and
\begin{equation}\label{32}
 |Q(\tau-i\gamma,\eta)U(0)|^2 \lesssim \frac{(\tau^2+\gamma^2+|\eta|^2)^s}{\gamma^{1+2s}}\|\hat{f}\|_{L_2(\R_+)}^2+ \frac{(\tau^2+\gamma^2+|\eta|^2)^s}{\gamma^{2s}}|\hat{g}|^2
\end{equation}
in the strictly hyperbolic case.

In the symmetric hyperbolic case we need to find an estimate for $|A^d_+U_+(0)|^2$. With $K=\left[ \begin{matrix} 0 & A_+ & 0 \end{matrix} \right] \Pi_{E^s}$, introduce
\[
 V(x_d) =  \frac{1}{2\pi i} \int_{\Gamma_-} e^{x_d\zeta}(\zeta A^d - G)^{-1} A^d d\zeta K^H(KK^H)^{-1} A_+U_+(0)\;,
\]
which is a solution to the equation (\ref{9}) with $\hat{f}=0$ and $A_+V_+(0)= A_+U_+(0)$.  Since $V(0)\in E^s(\tau-i\gamma,\eta) $ one gets from condition (\ref{5}) that
\begin{equation}\label{33}
 |BV(0)| \gtrsim \frac{\gamma^s}{(\tau^2+\gamma^2+|\eta|^2)^{s/2}}|A^d V(0)|
\end{equation}
because of homogeneity. Using now  $|BW| \lesssim |A^d W|$ because of $N(A^d)\subset N(B)$ and $BU(0)=\hat{g}$ we obtain, by applying the first inequality of Lemma \ref{F} to the function $U-V$ that,
\[
 \begin{split}
  |A_+ U_+(0)|^2& =  |A_+V_+(0)|^2\le |A^d V(0) |^2 \lesssim  \frac{(\tau^2+\gamma^2+|\eta|^2)^{s}}{\gamma^{2s}} |BV(0)|^2 \\
  & \lesssim  \frac{(\tau^2+\gamma^2+|\eta|^2)^{s}}{\gamma^{2s}} \big|B[U(0)-V(0)]\big|^2 + \frac{(\tau^2+\gamma^2+|\eta|^2)^{s}}{\gamma^{2s}} |BU(0)|^2 \\
    & \lesssim  \frac{(\tau^2+\gamma^2+|\eta|^2)^{s}}{\gamma^{2s}}\left|A^d[U(0)-V(0)]\right|^2+ \frac{(\tau^2+\gamma^2+|\eta|^2)^s}{\gamma^{2s}}|\hat{g}|^2 \\
    & \lesssim
 \frac{(\tau^2+\gamma^2+|\eta|^2)^s}{\gamma^{2s+1}}\|\hat{f}\|^2_{L_2(\R_+)}+ \frac{(\tau^2+\gamma^2+|\eta|^2)^s}{\gamma^{2s}}|\hat{g}|^2  \;.
 \end{split}
\]
The proof of (\ref{32}) uses the same idea with a few modifications which may be worthwhile to be pointed out. This time we set
\[
 V(x_d) = \frac{1}{2\pi i} \int_{\Gamma_-} e^{x_d \zeta}(\zeta A^d - G)^{-1} A^d d\zeta\, QU(0)\;,
\]
which is a solution to the the differential equation (\ref{9}) and $V(0) = \Pi_{E^s}QU(0)=QU(0) $. This time we use again the Kreiss-Sakamoto condition (\ref{33}) and the second estimate of Lemma \ref{F} applied to $U-V$.
\[
 \begin{split}
  |Q(\tau-i\gamma,\eta)U(0)|^2& =  |V(0)|^2 \lesssim  \frac{(\tau^2+\gamma^2+|\eta|^2)^{s}}{\gamma^{2s}} |BV(0)|^2
    \\
  & \lesssim  \frac{(\tau^2+\gamma^2+|\eta|^2)^{s}}{\gamma^{2s}} \big|B[U(0)-V(0)]\big|^2 + \frac{(\tau^2+\gamma^2+|\eta|^2)^{s}}{\gamma^{2s}} |BU(0)|^2 \\
    & \lesssim  \frac{(\tau^2+\gamma^2+|\eta|^2)^{s}}{\gamma^{2s}} |U(0)-V(0)|^2+ \frac{(\tau^2+\gamma^2+|\eta|^2)^s}{\gamma^{2s}}|\hat{g}|^2 \\
    & \lesssim
 \frac{(\tau^2+\gamma^2+|\eta|^2)^s}{\gamma^{2s+1}}\|\hat{f}\|^2_{L_2(\R_+)}+ \frac{(\tau^2+\gamma^2+|\eta|^2)^s}{\gamma^{2s}}|\hat{g}|^2  \;.
 \end{split}
\]
We integrate formula (\ref{12}) in $\tau$ and $\eta$ and obtain using Parseval's identity
\begin{multline*}
   \gamma \int_{R^{d+1}_+}|U(\tau-i\gamma,\eta,x_d)|^2d\tau d\eta dx_d + \int_{R^{d}}|A^d U(\tau-i\gamma,\eta,0)|^2d\tau d\eta\\ \lesssim \frac{1}{\gamma^{1+2s}}\|e^{-\gamma t}f\|^2_{s,\gamma} + \frac{1}{\gamma^{2s}}|e^{-\gamma t}g|^2_{s,\gamma}
\end{multline*}
for $\gamma\ge \underline{\gamma}$. This proves that $U \in L_2(\R_+,L_2(\R^{d}_{\tau,\eta}))$ for $\gamma\ge \underline{\gamma}$. Together with the analyticity of $U$ in $\tau-i\gamma$ in the lower half plane, this shows that $U$ is the Fourier transform of a function $e^{-\gamma t}u\in L_2(\R^{n+1}_+)$ for $\gamma \ge \underline{\gamma}$. Letting $\gamma \to \infty$ shows that $u$ vanishes for $t<0$.
\begin{remark}\label{D}
 The proof shows that the result of Theorem \ref{A} can be shifted to other Sobolev norms with respect to the tangential variables. For example, for $f\in L_2(\R^{d+1})$ and $g\in L_2(\R^d)$ supported in $t\ge 0$, there exists a unique solution $u\in L_{2}(\R_+,H_{loc}^{-s}(\R^d))$, vanishing for $t<0$, to the boundary problem (\ref{3}) and the estimate
  \[
   \gamma \|e^{-\gamma t}u\|^2_{-s,\gamma} + |e^{-\gamma t}A^d u(0)|^2_{-s,\gamma} \lesssim \frac{1}{\gamma^{1+2s}}\|e^{-\gamma t}f\|^2_{0} + \frac{1}{\gamma^{2s}}|e^{-\gamma t}g|^2_{0}
  \]
  holds for $\gamma$ sufficiently large.
\end{remark}
\setcounter{equation}{0}
\section{Examples}\label{iv}
\subsection{The wave equation with Neumann boundary condition}
 Consider the boundary value problem
 \begin{equation}\label{16}
  \partial_t^2 u - \Delta u = f \mbox{ in } \R^{d+1}_+=\{(t,x)\in \R^{d+1}\;:\;x_d>0\},\quad \partial_d u = g \mbox{ in } \{x_d=0\}
 \end{equation}
 where $\Delta$ denotes the Laplacian in $d\ge 2$ (space) variables. For the variable $w=e^{-\gamma t} u$ one obtains the equation
 \[
  (\partial_t +\gamma)^2 w - \Delta w = e^{-\gamma t} f\;.
 \]
 This second order equation can be reduced to a first order $2\times 2$ system which is strictly hyperbolic, see for example \cite[Chapter 7]{chazarainpiriou82}. With $v= (v_1,v_2)^T$ where $v_1 = \partial_d w$, $v_2= \Lambda w$, with
 \[
  \Lambda w = \frac{1}{(2\pi)^3} \int_{\R^d} e^{i(\tau t + y\cdot \eta)} \sqrt{\tau^2+\gamma^2+ |\eta|^2}\;\hat{w}(\tau,\eta,x_d) d\tau d\eta =\sqrt{\gamma^2 - \partial_t^2 - \Delta_y} \;w
 \]
 one obtains the $2\times 2$ system
 \[
  \partial_d v = \left[ \begin{matrix} 0 & [(\partial_t+\gamma)^2 - \Delta_y]\Lambda^{-1} \\ \Lambda & 0  \end{matrix} \right] v - \left[ \begin{matrix}e^{-\gamma t} f \\ 0 \end{matrix} \right]\;,
 \]
 where $\Delta_y$ the Laplacian in the spatially tangential variables. Since $A^d = I_2$, the stable subspace $E^s(\tau-i\gamma,\eta)$ of the matrix pair $(A^d,G)$ for $(\tau,\eta,\gamma) \in S^d_+$ is spanned by an eigenvector of the matrix
 \[
  G(\tau-i\gamma,\eta)=\left[ \begin{matrix} 0 & -(\tau-i\gamma)^2 + |\eta|^2 \\ 1 & 0  \end{matrix} \right]
 \]
 corresponding to the eigenvalue with negative real part. The eigenvalues are \\$\lambda = \pm\sqrt{|\eta|^2 -(\tau-i\gamma)^2}$ and in the sequel the square root symbol stands always for the root with positive real part. The eigenvector corresponding to the eigenvalue with negative real part is the vector
 \[
   z=\left[ \begin{matrix}\sqrt{|\eta|^2 -(\tau-i\gamma)^2}\\ -1 \end{matrix}\right]
 \]
 and $E^s(\tau-i\gamma,\eta)=\mbox{span}[z]$. There exists a constant $c$ such that $1 \le |z| \le c$ for all $(\tau,\eta,\gamma) \in S^d_+$. In order to establish the power of the Kreiss-Sakamoto condition (\ref{5}), observe that
 \[
  Bz = \sqrt{|\eta|^2 -(\tau-i\gamma)^2}\;.
 \]
 We claim that $|Bz| \gtrsim \gamma^{1/2}$ and that the power $1/2$ cannot be lowered. To prove this claim, use
 \[
  |\sqrt{|\eta|^2 -(\tau-i\gamma)^2}|^2 = ||\eta|^2 -(\tau-i\gamma)^2|
 \]
 and that for $(\tau,\eta,\gamma)\in S^d_+$,
 \begin{equation}\label{23}
  \begin{split}
    \left||\eta|^2 -(\tau-i\gamma)^2\right|^2 =& (\tau^2-\gamma^2-|\eta|^2)^2 + 4\gamma^2\tau^2  =  (\tau^2 - |\eta|^2)^2 +\gamma^4 + 2\gamma^2|\eta|^2 +   2\tau^2 \gamma^2 \\ \ge &  \gamma^2(\gamma^2 + 2\tau^2 + 2|\eta|^2)  \ge    \gamma^2
  \end{split}
 \end{equation}
 where the power in the last inequality is optimal. Hence, $|Bv|\gtrsim \gamma^{1/2}|v|$ for all $v\in E^s$.
 Applying Theorem \ref{A} one concludes that for $f\in L_2(\R_+,H^{1/2}(\R^d))$ and $g\in H^{1/2}(\R^d)$ there exists a unique solution $u\in H_{loc}^{1}(\R^{d+1})$ to the boundary value problem (\ref{16}) and $u\Big|_{\{x_d=0\}}\in H^1_{loc}(\R^d)$. Or, relying on Remark \ref{D} we infer that for $f\in L_2(\R^{d+1}_+)$ and $g\in L_2(\R^d)$ there exists a unique solution $u\in H_{loc}^{1/2}(\R^{d+1})$ to the boundary value problem (\ref{16}) and $u\Big|_{\{x_d=0\}}\in H^{1/2}_{loc}(\R^d)$.

 As mentioned in the introduction, this result is stronger than the general result for hyperbolic problems with conservative boundary condition \cite{el12} applied to this problem. Moreover, the regularity of the trace $u\Big|_{\{x_d=0\}}\in H^{1/2}_{loc}(\R^d)$ is even stronger than the results for the wave equation derived more than 20 years ago by Lasiecka and Triggiani \cite[Main Theorem 1.3]{lt90}. On the other hand, this work contains better interior regularity results.
\subsection{The oblique derivative problem for the wave equation}
We consider the problem
\begin{equation}\label{17}
  \partial_t^2 u - \Delta u = f \mbox{ in } \R^{d+1}_+=\{(t,x)\in \R^{d+1}\;:\;x_d>0\},\quad \partial_d u+b\cdot \nabla_y u = g \mbox{ in } \{x_d=0\}
\end{equation}
where $b\in \R^{d-1}$, $d\ge 2$, $b\neq 0$, and $\nabla_y u = (\partial_{x_1} u,..., \partial_{x_{d-1}}u)$ is the spatially tangential gradient. For a large part the analysis is the same as in the previous examples. However, this time we have
\[
 Bz = \sqrt{|\eta|^2 -(\tau-i\gamma)^2} + i b\cdot\eta
\]
and we will show that we have the estimate $|Bz|\gtrsim \gamma$ which cannot be improved.
Following Chazarain and Piriou \cite[section 13]{cp72} one estimates
\[
 |Bz| \ge |\Re Bz| \gtrsim \gamma\;,
\]
which relies on the inequality $\Re \sqrt{c^2 - (a-di)^2}\ge d$ for $a,c,d \in \R$ and $d\ge 0$. To prove this inequality, set
\[
 \sqrt{c^2 - (a-di)^2} = \alpha +\beta i,\quad \alpha\ge 0\;,
\]
which implies $c^2+d^2-a^2 +2ad i = \alpha^2-\beta^2 +2\alpha\beta i$ and hence $ad=\alpha\beta$. If $a=0$, there is nothing to prove. To prove $\alpha\ge d$ when $a\neq 0$ we argue by contradiction. Suppose that $0\le \alpha < d$. Then $\beta>a>0$ or $\beta<a< 0$ and $\beta^2 - \alpha^2 > a^2 -d^2$ which contradicts $\beta^2 - \alpha^2 = a^2 -d^2 -c^2 \le a^2-d^2$.

Next we show that this estimate cannot be improved. Suppose that $|Bz| \gtrsim \gamma^\theta$ for some $\theta \in [0,1)$. Choose a sequence $\{z_n\} \subset E^s(\tau-i\gamma,\eta)$ such that $z_n=(\tau_n,\eta_n,\gamma_n) \in S^d_+$, $(b\cdot \eta_n)^2 = \tau_n^2-|\eta_n|^2\gtrsim 1$, $b\cdot\eta_n<0$, $\tau_n>0$, and $\gamma_n\to 0$ for $n\to \infty$. Then
\[
 \begin{split}
  \lim_{n\to \infty}\frac{Bz_n}{\gamma_n^\theta}&=\lim_{n\to \infty} \frac{\sqrt{|\eta_n|^2 -(\tau_n-i\gamma_n)^2} + i b\cdot\eta_n}{\gamma^\theta_n} = \lim_{n\to \infty} \frac{|\eta_n|^2 -(\tau_n-i\gamma_n)^2 + (b\cdot\eta_n)^2}{(\sqrt{|\eta_n|^2 -(\tau_n-i\gamma_n)^2} - i b\cdot\eta_n)\gamma^\theta_n} \\
  &= \lim_{n\to \infty} \frac{\gamma_n^{1-\theta}(2\tau_ni+\gamma_n)}{\sqrt{|\eta_n|^2 -(\tau_n-i\gamma_n)^2} - i b\cdot\eta_n}=0
 \end{split}
\]
which shows that the estimate $|Bz|\gtrsim \gamma$ is optimal.
\subsection{The isotropic Maxwell equations}
Let $(E,H)$ denote the electromagnetic field in space time and let $\ve$ and $\mu$ be positive constants which denote the electric permittivity and the magnetic permeability of the underlying medium, respectively. The electromagnetic field satisfies Maxwell's equations which is a $6\times 6$ hyperbolic system of first order,
\begin{equation}\label{19}
 \begin{split}
  \ve \partial_t E - \nabla\times H &= f_1 \\ \mu \partial_t H + \nabla\times E &= f_2
 \end{split}
  \qquad \mbox{ in } \R^4_+=\{ (t,x)\in \R^4 \;:\; x_3>0 \}
\end{equation}
The forcing term $f=(f_1,f_2)$ represents current densities, usually $f_2\equiv 0$. The boundary condition is of the form
\begin{equation}\label{20}
 \nu \times E = g \mbox{ in } \{x_3 = 0\}
\end{equation}
and models a current flux along the boundary of the medium. Here $\nu=(0,0,-1)^T$ is the exterior unit normal vector and
\begin{equation}\label{34}
 A^3\left[\begin{matrix} e \\ h \end{matrix} \right] = \left[\begin{matrix} -\nu \times h \\ \nu\times e \end{matrix} \right]\;.
\end{equation}
Note that the boundary is characteristic since $\det  A^3 = 0$ and that $N(A^3)\subset N(B)$. It is known that the boundary condition does not satisfy the Kreiss-Sakamoto condition \cite[Section 2]{mo75},\cite{el09}. However, it is a conservative boundary condition and there is a well-posedness result with a loss of one derivative on the boundary \cite[Theorems 1.2,1.3,1.4]{el12}.

Our aim is to show that condition (\ref{5}) is satisfied with $s=1$ and that $s$ cannot be smaller. To determine the stable subspace $E^s(\tau-i\gamma,\eta)$ for $(\tau,\eta,\gamma) \in S^3_+$ consider the solutions to the system
\[
 \begin{split}
  (\tau - i\gamma ) \ve e - \zeta \times h &= 0 \\ (\tau - i\gamma ) \mu h + \zeta \times e &= 0 \;.
 \end{split}
\]
where $\zeta = (\eta_1,\eta_2,-i\xi)^T$. Hence,
\[
 e= \frac{\zeta \times h}{(\tau -i\gamma)\ve} \;,\quad  h =-\frac{\zeta \times e}{(\tau-i\gamma)\mu} = \frac{ (\zeta\cdot\zeta) h}{(\tau-i\gamma)^2\ve\mu}\;, \quad \ \zeta\cdot h  = 0\;, \quad  \zeta\cdot e  = 0\;,
\]
which gives $\ve\mu (\tau-i\gamma)^2 = \zeta\cdot \zeta$. The (double) eigenvalue with negative real part is $\xi = -\sqrt{|\eta|^2 - \ve\mu(\tau-i\gamma)^2}$ and the stable subspace is the two-dimensional vector space
\[
 \begin{split}
  E^s(\tau-i\gamma,\eta) &= \left\{ \left[ \begin{matrix} e \\ h \end{matrix}\right] \in \C^6 \;:\; \zeta\cdot h  = 0,\; \ve e= \frac{\zeta\times h}{\tau - i \gamma},\;\ve\mu (\tau-i\gamma)^2 = |\eta|^2 - \xi^2 \right\}\\
  &=\left\{ \left[ \begin{matrix} e \\ h \end{matrix}\right] \in \C^6 \;:\;  \zeta \cdot e  = 0,\; \mu h=- \frac{\zeta\times e}{\tau - i \gamma},\;\ve\mu (\tau-i\gamma)^2 = |\eta|^2 - \xi^2\right\}\;.
 \end{split}
\]
An orthogonal basis of the stable subspace is given by
\[
 w_1 = \left[ \begin{matrix} \mu(\tau - i \gamma) \zeta^\perp \\ -\zeta \times \zeta^\perp \end{matrix} \right]\;, \qquad
 w_2 = \left[ \begin{matrix} \zeta \times \zeta^\perp \\ \ve(\tau - i \gamma) \zeta^\perp \end{matrix} \right]\;,
\]
where $\zeta^\perp$ is a real unit vector which is perpendicular to both $\Re \zeta$ and $\nu$. There exists two constants $c_1$ and $c_2$ such that $c_1 \le |w_1| \le c_2$ and $c_1 \le |w_2| \le c_2$ for all $(\tau,\eta,\gamma)\in S^3_+$. We will show that
\begin{equation}\label{22}
 |B(\alpha w_1 + \beta w_2)| \ge \gamma|A^3(\alpha w_1 + \beta w_1)|
\end{equation}
where $\alpha,\beta \in \C$ such that $|\alpha|^2 + |\beta|^2 =1$ and $(\tau,\eta,\gamma)\in S^3_+$. This establishes the Kreiss-Sakamoto condition with power $s=1$.

To prove (\ref{22}), compute as in (\ref{23})
\[
 \begin{split}
 |B(\alpha w_1 + \beta w_2)|^2 &= |\alpha \mu (\tau-i\gamma) \nu\times \zeta^\perp + \beta \zeta^\perp (i\xi)|^2 = |\alpha|^2 \mu^2  (\tau^2+\gamma^2) + |\beta|^2 |\xi|^2 \\& = |\alpha|^2(\tau^2+\gamma^2) + |\beta|^2\sqrt{\big||\eta|^2-\ve\mu(\tau-i\gamma)^2\big|^2}\\ & \gtrsim |\alpha|^2(\gamma^2+\tau^2) + |\beta|^2\left[ \big|\ve\mu \tau^2 -|\eta|^2\big| + \gamma\right] \gtrsim \gamma^2 \;,
 \end{split}
\]
and using (\ref{34}), compute
\[
 \begin{split}
  |A^3(\alpha w_1 + \beta w_2)|^2&=(\tau^2+\gamma^2)[\ve^2|\beta|^2+ \mu^2|\alpha|^2] + (|\alpha|^2+|\beta|^2)|\xi|^2\eqsim \tau^2 +\gamma^2 + |\xi|^2 \\ &= \tau^2+\gamma^2 +\sqrt{\big||\eta|^2-\ve\mu(\tau-i\gamma)^2\big|^2}\\
  &\eqsim \tau^2 + \gamma^2 + \left[ \big|\ve\mu \tau^2 -|\eta|^2\big| + \gamma\right] \eqsim 1\;.
 \end{split}
\]
where we write $a\eqsim b$ if and only if $a\lesssim b$ and $a\gtrsim b$. This estimate cannot be improved.

Applying Theorem \ref{A} we know that for $f\in H^{1}(\R^4_+)$, $g\in H^{1}(\R^3)$ which vanish for $t<0$, there exists a unique solution $(E,H)\in L_{2,loc}(\R^4_+)$ to the boundary value problem (\ref{19}),(\ref{20}) and $\nu\times H\Big|_{\{x_3=0\}} \in L_{2,loc}(\R^3)$. This result is certainly not better than the aforementioned results by the author which apply even to variable coefficients. Our work here shows that the results in \cite{el12} are sharp in the sense that a loss of one derivative on the boundary occurs.

\bibliographystyle{plain}
\bibliography{article,mylib}
\end{document}